\documentclass{amsart}

\usepackage{amsfonts}
\usepackage{amssymb}
\usepackage{latexsym}

\newtheorem{theorem}{Theorem}[section]
\newtheorem{lemma}[theorem]{Lemma}

\theoremstyle{definition}
\newtheorem{definition}[theorem]{Definition}
\newtheorem{example}[theorem]{Example}
\newtheorem{proposition}[theorem]{Proposition}

\newtheorem{corollary}[theorem]{Corollary}

\theoremstyle{remark}

\numberwithin{equation}{section}

\newcommand{\R}{\mathbb{R}}
\newcommand{\Z}{\mathbb{Z}}
\newcommand{\F}{\mathbb{F}}
\newcommand{\Imt}{\mbox{Im}\,}

\begin{document}

\title
{Classifying closed 2-orbifolds with Euler characteristics}

\author{Whitney DuVal}
\address{Department of Mathematics and Computer Science,
Rhodes College, 2000 N. Parkway, Memphis, TN 38112}
\email{DuvWR@rhodes.edu}

\author{John Schulte}
\email{SchJM@rhodes.edu}

\author{Christopher Seaton}
\email{seatonc@rhodes.edu}

\author{Bradford Taylor}
\email{TayBP@rhodes.edu}

\subjclass[2000]{Primary 57R20, 57S17; Secondary 22A22, 57P99}

\begin{abstract}

We determine the extent to which the collection of
$\Gamma$-Euler-Satake characteristics classify closed $2$-orbifolds.
In particular,  we show that the closed, connected, effective,
orientable $2$-orbifolds are classified by the collection of
$\Gamma$-Euler-Satake characteristics corresponding to free or free
abelian $\Gamma$ and are not classified by those corresponding to
any finite collection of
finitely generated discrete groups.  Similarly, we show that such a classification is
not possible for non-orientable $2$-orbifolds and any collection of
$\Gamma$, nor for noneffective $2$-orbifolds.  As a corollary, we
generate families of orbifolds with the same $\Gamma$-Euler-Satake
characteristics in arbitrary dimensions for any finite collection of
$\Gamma$; this is used to demonstrate that the $\Gamma$-Euler-Satake
characteristics each constitute new invariants of  orbifolds.

\end{abstract}

\maketitle


\section{Introduction}
\label{sec-intro}

In a recent paper \cite{farsiseaton1}, the third author and Carla
Farsi introduced the \emph{$\Gamma$-sectors of an orbifold $Q$}, a
generalization of the inertia orbifold of $Q$ that is defined for
any finitely generated discrete group $\Gamma$.  In this context,
the inertia orbifold (originally
defined by Kawasaki in \cite{kawasaki1}; see also
\cite{ademleidaruan} and \cite{chenruanorbcohom}) corresponds to the
case $\Gamma = \Z$; similarly, the $k$-multi-sectors of Chen and
Ruan (see \cite{ademleidaruan} or \cite{chenruanorbcohom})
correspond to the case when $\Gamma = \F_k$ is the free group with
$k$ generators.

In \cite{farsiseaton3}, it is shown that several Euler
characteristics that have been defined for orbifolds correspond to
the \emph{$\Gamma$-Euler-Satake characteristics} for some choice of
$\Gamma$, i.e. the Euler-Satake characteristic of the
$\Gamma$-sectors of $Q$, denoted $\chi_\Gamma^{ES}(Q)$. Hence, the
$\Gamma$-sectors offer a framework in which to generalize the Euler
characteristics of Bryan and Fulman (see \cite{bryanfulman}) and
Tamanoi (see \cite{tamanoi1} and \cite{tamanoi2}) to closed
orbifolds that are not necessarily global quotients.  In this
context, the Euler characteristics of Bryan and Fulman correspond to
$\Gamma$ free abelian; in particular, the stringy orbifold Euler
characteristic defined for global quotients in \cite{dixon} and for
general orbifolds in \cite{roan} corresponds to the case $\Gamma =
\Z^2$.

Here, we address the question of whether the
$\Gamma$-Euler-Satake characteristics classify closed, connected,
2-dimensional orbifolds. The diffeomorphism-types of all closed
$2$-orbifolds are well-known; see e.g. \cite{thurston} or
\cite{boileaump}.  Here, however, we express this classification in
a framework generalizing the familiar classification of closed
2-manifolds.  An additional motivation of this investigation is to
explore the extent to which the $\Gamma$-Euler-Satake
characteristics constitute new invariants for orbifolds.  Indeed, from their
definition, the degree to which collections of the $\Gamma$-Euler-Satake characteristics
depend on one another is unclear.  We will see, however, that the characteristics
corresponding to abelian $\Gamma$ are in some sense independent; the class of
$2$-dimensional orientable orbifolds is sufficiently large to illustrate this fact.
In this case, each $\Gamma$-Euler-Satake characteristic corresponds to the
$\Gamma^\prime$-Euler-Satake characteristic for an abelain $\Gamma^\prime$.
In the future, we will investigate classes of orbifolds that may
indicate the differences between abelian and nonabelian
$\Gamma$.

To simplify notation, for a closed orbifold $Q$, we define
\[
    \chi_{(l)}^{ES}(Q) = \chi_{\Z^l}^{ES}(Q).
\]
Then for $l \geq 1$, these Euler characteristics correspond to the orbifold Euler
characteristics defined for global quotients in \cite{bryanfulman}
(note that our $\chi_{(l)}^{ES}(Q)$ corresponds to $\chi_{l+1}(M, G)$ in
\cite{bryanfulman} when $Q$ is given by the action of a finite group $G$
on a manifold $M$).  It is observed in \cite[Section 4.1]{farsiseaton3} that
$\chi_{(l)}^{ES}(Q)$ also corresponds to the Euler-Satake characteristic
of the $l$th inertia orbifold of $Q$ and the Euler characteristic of the
(underlying topological space of the) $l-1$st inertia orbifold.

If $Q$ is an \emph{abelian orbifold} (i.e. all isotropy groups of $Q$ are abelian),
it is easy to see that
\[
    \chi_{(l)}^{ES}(Q) = \chi_{\F_l}^{ES}(Q),
\]
where $\F_l$ denotes the free group with $l$ generators; in particular, this follows
from Lemma \ref{lem-fingenmodcommutators} below.  It follows that in this case,
$\chi_{(l)}^{ES}(Q)$ is the Euler-Satake characteristic of the $l$-multi-sectors of $Q$;
see \cite{ademleidaruan}.

Of primary interest will be the case of a
closed, connected, effective, orientable $2$-orbifold $Q$,
for which the $\chi_{(l)}^{ES}(Q)$ will play a dominant role.
Our first main result is a positive classification of these
orbifolds using the $\chi_{(l)}^{ES}$.

\begin{theorem}
\label{thrm-mainpositive}

Let $Q$ and $Q^\prime$ be closed, connected, effective, orientable
$2$-orbifolds such that $\chi_{(l)}^{ES}(Q) = \chi_{(l)}^{ES}(Q^\prime)$
for each nonnegative integer $l$. Then $Q$ and $Q^\prime$ are
diffeomorphic.

\end{theorem}

It is well-known (see e.g. \cite{massey}) that closed, connected,
orientable, 2-dimensional manifolds are completely characterized by
their Euler characteristic.  If $Q$ is a manifold, the
$\Gamma$-Euler-Satake characteristic of $Q$ reduces to the usual
Euler characteristic for any $\Gamma$. Hence, Theorem
\ref{thrm-mainpositive} constitutes a generalization of this result
to orbifolds. However, this class of orbifolds is large enough to
produce the following.

\begin{theorem}
\label{thrm-mainnegative}

Let $N \geq 2$ be an integer and let $\mathfrak{G}$ be any finite collection of finitely
generated discrete groups. Then there are distinct closed, connected, effective, orientable $2$-orbifolds
$Q_1, Q_2, \ldots , Q_N$ such that for each $\Gamma \in
\mathfrak{G}$,
\[
    \chi_\Gamma^{ES}(Q_1) = \chi_\Gamma^{ES}(Q_2) = \cdots = \chi_\Gamma^{ES}(Q_N).
\]

\end{theorem}

It follows that the classification of Theorem \ref{thrm-mainpositive} cannot be
improved upon using the $\Gamma$-Euler-Satake characteristics.  Note
that Theorem \ref{thrm-mainnegativegeneral} is a slightly more general
version of Theorem \ref{thrm-mainnegative}, though clumsier to
state.

The outline of this work is as follows.  In Section
\ref{sec-definitions}, we recall the necessary definitions and
summarize the pertinent preliminary material.  We study effective, orientable
$2$-orbifolds in Section
\ref{sec-GammaESCcleffornt2orb} and prove Theorems \ref{thrm-mainpositive} and
\ref{thrm-mainnegative}.  In Section \ref{sec-otherorbs}, we
demonstrate through examples that the hypotheses of Theorem
\ref{thrm-mainpositive} cannot be relaxed.


This paper is the result of the course `Topics: Orbifold Euler
Characteristics' taught in the Rhodes College Mathematics and
Computer Science Department in the Fall of 2008. We express our
appreciation to the department and college for the versatility and
support that allowed us to hold this seminar and explore these
results.  We would also like to thank Rachel Dunwell for helpful
suggestions and assistance.

The third author would like to thank Carla Farsi and Anna Casteen,
with whom he has conducted work leading to this project.
In particular, Proposition \ref{prop-orientedECformula} was first
proved by Anna Casteen as part of her senior seminar project `Finding orbifold
Euler characteristics' at Rhodes College in the spring of 2008.


\section{Background and Definitions}
\label{sec-definitions}

In this section, we briefly introduce the required definitions and
fix notation. For a more thorough background on orbifolds, the
reader is referred to \cite{ademleidaruan} or \cite{chenruangwt};
see also \cite{boileaump}, \cite{moerdijkmrcun}, or \cite{thurston},
and note that the orbifolds in these latter references correspond to
effective orbifolds.  We will have the occasion to consider
noneffective orbifolds only in Example \ref{ex-noneffective} and
only in the form of a global quotient.

An \emph{orbifold} $Q$ is most succinctly defined to be a Morita
equivalence class of orbifold groupoids, i.e. proper \'{e}tale Lie groupoids.  Such a
groupoid $\mathcal{G}$ is called a \emph{presentation} of the
orbifold $Q$, and two orbifold groupoids $\mathcal{G}$ and
$\mathcal{G}^\prime$ present the same orbifold if and only if they
are Morita equivalent. In this case, their orbit spaces
$|\mathcal{G}|$ and $|\mathcal{G}^\prime|$ are naturally
homeomorphic, and we say that they are \emph{diffeomorphic} as
orbifolds.

Fix a proper \'{e}tale Lie groupoid $\mathcal{G}$ with space of
objects $G_0$ and space of arrows $G_1$.  For each $x \in G_0$,
there is a neighborhood $V_x \subseteq G_0$ of $x$ diffeomorphic to
$\R^n$ such that if $G_x$ denotes the isotropy group of $x$, then
there is a $G_x$-action on $V_x$, and the restriction
$\mathcal{G}|_{V_x}$ is isomorphic as a Lie groupoid to the
translation groupoid $G_x \ltimes V_x$.  We let $\pi_x : V_x
\rightarrow |\mathcal{G}|$ denote the quotient map into the orbit space of
$\mathcal{G}$.  In this way, the definition of an orbifold in terms
of orbifold charts is recovered, as $\{ V_x, G_x, \pi_x \}$ gives an
orbifold chart for $Q$ near the point representing the orbit of $x$.
Note that we can always take $x$ to correspond to the origin in
$\R^n$ and $G_x$ to act linearly; we then refer to $\{V_x, G_x, \pi_x\}$
as a \emph{linear chart}.  If $y$ is another point in $G_0$
in the orbit of $x$, then $G_y$ and $G_x$ are isomorphic.  Hence, if
$p \in |\mathcal{G}|$ denotes the orbit of $x$, then we can define $G_p$ to be
(the isomorphism class of) $G_x$. The point $p \in |\mathcal{G}|$ is a
nonsingular point if $G_p$ is trivial and a singular point
otherwise.

We say that an orbifold $Q$ is \emph{effective} if $\mathcal{G}$ is
an effective groupoid, or equivalently if the local $G_x$-actions on
the $V_x$ are effective.  By \emph{closed} or \emph{connected}, we
mean that the orbit space $|\mathcal{G}|$ is compact or connected,
respectively, as a topological space.  An orbifold is
\emph{oriented} if $G_0$ is equipped with a $G_1$-invariant
orientation; if $\mathcal{G}$ admits an orientation, we say that $Q$
is \emph{orientable}.  Note that each of these qualities is
preserved under Morita equivalence so that they describe the
orbifold $Q$ as well as the presentation $\mathcal{G}$.

If $Q$ is a closed, connected, effective, $2$-dimensional orbifold
and $x \in G_0$, then $G_x$ is a finite subgroup of $O(2)$ (with
respect to any inner product on $V_x$). It follows that $G_x$ is either a
cyclic group acting as rotations, a group isomorphic to $\Z/2\Z$
acting as reflection through a line, or a group isomorphic to a
dihedral group whose action is generated by reflections through two
lines (see \cite{thurston}).  The singular points associated to
these actions are referred to as \emph{cone points} (or
\emph{elliptic points}), \emph{reflector lines}, and \emph{corner
reflectors}, respectively. Only the first of these three preserves
an orientation of $\R^2$; hence, if we assume further that $Q$ is
orientable, then the singular points are isolated cone points with
cyclic isotropy.  By the order of the cone point, we will mean the
order of the isotropy group.  It follows that the
underlying space is homeomorphic to a closed, connected orientable surface,
and the set of singular points corresponds to a finite collection
$\{p_1, p_2, \ldots , p_k \}$ of cone points.

A closed, connected, effective, orientable, $2$-dimensional orbifold,
then, is determined by the genus $g$ of the underlying space, a
nonnegative integer $k$ indicating the number of cone points, and an
integer $m_i \geq 2$ for $i = 1, 2, \ldots , k$ with
$m_1 \leq m_2 \leq \cdots \leq m_k$, indicating the
order of each cone point.  We will use the notation $\Sigma_g(m_1,
\ldots , m_k)$ to denote this orbifold.  Note that we will often
refer to the genus of the underlying space of $Q$ simply as the
genus of $Q$.

Let $Q$ be an orbifold.  For each finitely generated discrete group
$\Gamma$, we associate to $Q$ an orbifold $\tilde{Q}_\Gamma$ called
the \emph{$\Gamma$-sectors} of $Q$.  We recall this construction
briefly; see \cite{farsiseaton1} for more details. Let
$\mathcal{S}_\mathcal{G}^\Gamma$ denote the space of groupoid homomorphisms
from $\Gamma$ into $\mathcal{G}$ or equivalently group homomorphisms from
$\Gamma$ into an isotropy group $G_x$ of $\mathcal{G}$.  Then
$\mathcal{S}_\mathcal{G}^\Gamma$ has the structure of a smooth
manifold, possibly with connected components of different dimensions.  There
is a natural $\mathcal{G}$-action on
$\mathcal{S}_\mathcal{G}^\Gamma$ by pointwise conjugation, and the
groupoid $\mathcal{G}^\Gamma = \mathcal{G} \ltimes
\mathcal{S}_\mathcal{G}^\Gamma$ is a presentation of the orbifold of
$\Gamma$-sectors $\tilde{Q}_\Gamma$. If $\{ V_x, G_x, \pi_x \}$ is a
linear chart for $Q$ and $\phi_x : \Gamma \rightarrow G_x$ is a
homomorphism, then $\{ V_x^{\langle \phi_x \rangle},
C_{G_x}(\phi_x), \pi_x^{\phi_x} \}$ is a linear chart for
$\tilde{Q}_\Gamma$ near $\phi_x$ where $V_x^{\langle \phi_x
\rangle}$ denotes the subspace of $V_x$ fixed by the image of
$\phi_x$, $C_{G_x}(\phi_x)$ is the centralizer of the image of
$\phi_x$ in $G_x$, and $\pi_x^{\phi_x}$ denotes the quotient map of the
$C_{G_x}(\phi_x)$-action.  The connected component $\tilde{Q}_{(1)}$
of $\tilde{Q}_\Gamma$ corresponding to the identity homomorphism
(into any isotropy group) is diffeomorphic to $Q$.  We denote the
connected component of a homomorphism $\phi_x : \Gamma \rightarrow
G_x$ by $\tilde{Q}_{(\phi)}$. Note that the $\Z$-sectors correspond
to the inertia orbifold, and the $\F_l$-sectors correspond to the
$l$-multi-sectors (see \cite{farsiseaton2}; see also
\cite{ademleidaruan} or \cite{chenruanorbcohom} for the
definitions).

In the case that $Q$ is presented by $M \rtimes G$ where $G$
is a finite group acting on the smooth manifold $M$, then our description of the
$\Gamma$-sectors corresponds to that of Tamanoi in \cite{tamanoi1} and \cite{tamanoi2},
where
\begin{equation}
\label{eq-globalsectordescrip}
    \tilde{Q}_\Gamma = 
    \coprod\limits_{(\phi) \in \mbox{\scriptsize HOM}(\Gamma, G)/G}
    M^{\langle \phi \rangle} \rtimes C_G(\phi).
\end{equation}
Here, the union is over conjugacy classes $(\phi)$ of homomorphisms
$\phi \in \mbox{HOM}(\Gamma, G)$.  In this case, we use $(M;G)_{(\phi)}$ to denote
$M^{\langle \phi \rangle} \rtimes C_G(\phi)$.  Note that this description coincides
with ours more generally for $G$ a Lie group with certain restrictions on the action;
see \cite[Section 3]{farsiseaton2}.

The \emph{Euler-Satake characteristic} was first defined in
\cite{satake2}, then called the \emph{Euler characteristic as a
$V$-manifold}.  Satake's definition generalizes directly to the
noneffective case.  Given a simplicial decomposition $\mathcal{T}$
of the underlying space of $Q$ such that the order of the isotropy
group $G_\sigma$ on the interior of each simplex $\sigma
\in \mathcal{T}$ is constant (which always exists; see
\cite{moerdijkpronksimplicial} or \cite{farsiseaton3}), we define
\[
    \chi_{ES}(Q)
    =
    \sum\limits_{\sigma \in \mathcal{T}}
    \frac{(-1)^{\mbox{\scriptsize dim}\: \sigma}}{|G_\sigma|}.
\]
The Euler-Satake characteristic clearly reduces to the usual Euler
characteristic in the case that each isotropy group of $Q$ is
trivial, i.e. in the case of a manifold.

Given a finitely generated discrete group, we define the
\emph{$\Gamma$-Euler-Satake characteristic of $Q$} to be the
Euler-Satake characteristic of the $\Gamma$-sectors of $Q$; i.e.
\[
    \chi_\Gamma^{ES}(Q)
    =
    \chi_{ES}\left(\tilde{Q}_\Gamma \right).
\]
The Euler-Satake characteristic of a disconnected orbifold is of
course equal to the sum of the Euler-Satake characteristics of the
connected components.  See \cite{farsiseaton3} for properties of the
Euler-Satake characteristic and $\Gamma$-Euler-Satake
characteristics.


\section{The $\Gamma$-Euler-Satake Characteristics of Effective, Orientable
        $2$-Orbifolds}
\label{sec-GammaESCcleffornt2orb}

In this section, we restrict our attention to closed, connected,
effective, orientable $2$-orbifolds.  In Subsection
\ref{subsec-formulaconstruc}, we determine a formula for the $l$th
Euler-Satake characteristics in this case and use this formula to prove Theorem
\ref{thrm-mainpositive}. In Subsection \ref{subsec-negativeresults},
we construct for each finite collection of nonnegative integers
$l$ an arbitrarily large
(finite) collection of orbifolds such that the $l$th Euler-Satake
characteristics coincide. In Subsection \ref{subsec-generalgamma},
we generalize to arbitrary $\Gamma$, proving Theorem
\ref{thrm-mainnegative}.


\subsection{The Classification for Free Abelian $\Gamma$}
\label{subsec-formulaconstruc}

Let $Q$ be a closed, connected, effective, orientable $2$-orbifold.
As mentioned in Section \ref{sec-definitions}, $Q$ is of the form
$\Sigma_g(m_1, \ldots , m_k)$ for some nonnegative integers $g$ and
$k$ and integers $2 \leq m_1 \leq m_2\leq \cdots \leq m_k$. Let
$\mathcal{G}$ be an orbifold groupoid presenting $Q$.  We begin by
describing $\tilde{Q}_\Gamma$ in this case.

Given a finitely generated discrete group $\Gamma$, a homomorphism
$\phi_x : \Gamma \rightarrow \mathcal{G}$ corresponds to a choice of
a point $x$ in an orbifold chart $\{ V_x, G_x, \pi_x \}$ for $Q$ and
a homomorphism $\Gamma \rightarrow G_x$, which we also denote
$\phi_x$. If $\phi_x$ is trivial so that its image is the trivial
group, then it is on the same connected component as all such
homomorphisms, and $\tilde{Q}_{(\phi)}$ is diffeomorphic to $Q$.
Otherwise, $\pi_x(x) = p_i$ is one of the singular points of $Q$,
and $\phi_x$ corresponds to a nontrivial homomorphisms into
$\Z/m_i\Z$ acting on $V_x = \R^2$ by rotations.  It follows that the
$(\Imt \phi_x)$-fixed-point subset of $\R^2$ consists of a single
point $x$, and $\phi_x$ is the only point in the connected component
$\tilde{Q}_{(\phi)}$ of $\tilde{Q}_\Gamma$.  A chart for
$\tilde{Q}_{(\phi)}$ is of the form $\{ V_x^{\langle \phi_x
\rangle}, C_G(\phi_x), \pi_x^{\phi_x} \} = \{ \{x\}, \Z/m_i\Z,
\pi_x^{\phi_x} \}$, so that $\tilde{Q}_{(\phi)}$ is a point equipped
with the trivial action of $\Z/m_i\Z$.  As the local groups of $Q$
are abelian, and as the singular points of $Q$ are isolated, the
$\mathcal{G}$-orbits of
nontrivial homomorphisms $\phi_x$ are trivial.  Hence, for each cone
point $p_i$ with isotropy group $\Z/m_i\Z$, there are exactly
$|\mbox{HOM}(\Gamma, \Z/m_i\Z)| - 1 = m_i - 1$ connected components
corresponding to ${p_i}$ with trivial $\Z/m_i\Z$-action.

We use these observations to derive the following, which gives a
formula for the $l$th Euler-Satake characteristic of a closed, connected
effective, orientable $2$-orbifold.

\begin{proposition}
\label{prop-orientedECformula}

Let $Q = \Sigma_g(m_1, \ldots , m_k)$ be a closed, connected, effective,
orientable $2$-orbifold with notation as above.  Then for each integer $l \geq 0$,
\begin{equation}
\label{eq-lESchar}
    \chi_{(l)}^{ES}(Q) = 2 - 2g - k + \sum\limits_{i=1}^k m_i^{l - 1}.
\end{equation}

\end{proposition}

\begin{proof}

Let $\mathcal{T}$ be a simplicial decomposition of $Q$ subordinate to the singular strata
(see \cite{moerdijkpronksimplicial} or \cite{farsiseaton2});
in this context, this means simply that each singular point $p_i$ corresponds to a vertex
of $\mathcal{T}$.  Then
\[
\begin{array}{rcl}
    \sum\limits_{\sigma \in \mathcal{T}} (-1)^{\mbox{\scriptsize dim}\: \sigma}
        &=&
        \chi_{top} (Q)      \\
        &=&
        2 - 2g
\end{array}
\]
where $\chi_{top}(Q)$ denotes the usual Euler characteristic of the underlying space of $Q$.
It follows that
\[
\begin{array}{rcl}
    \chi_{(0)}^{ES}(Q)
        &=&
        \chi_{ES}(Q)
            \\
        &=&
        \sum\limits_{\sigma \in \mathcal{T}} (-1)^{\mbox{\scriptsize dim}\: \sigma}
        - k + \sum\limits_{i=1}^k \frac{1}{m_i}
            \\
        &=&
        2 - 2g - k + \sum\limits_{i=1}^k \frac{1}{m_i}.
\end{array}
\]

Now, let $l \geq 0$ be an integer.  Each cone point $p_i$
corresponds to $|\mbox{HOM}(\Z^l, \Z/m_i\Z)| - 1 = m_i^l - 1$
identical $\Z^l$-sectors, each given by a single point equipped with
the trivial action of $\Z/m_i\Z$.  It follows that the Euler-Satake
characteristic of the corresponding $\Gamma$-sector is
$\frac{1}{m_i}$, and hence
\[
\begin{array}{rcl}
    \chi_{(l)}^{ES}(Q)
        &=&
        2 - 2g - k + \sum\limits_{i=1}^k \frac{1}{m_i}
        + \sum\limits_{i=1}^k (m_i^l - 1) \frac{1}{m_i}
        \\\\
        &=&
        2 - 2g - k + \sum\limits_{i=1}^k m_i^{l - 1},
\end{array}
\]
completing the proof.

\end{proof}

Note that as $\chi_{(0)}^{ES}(Q) = \chi_{ES}(Q)$, the case $l = 0$ of Equation \ref{eq-lESchar}
coincides with \cite[Equation 13.3.4]{thurston} for orientable
orbifolds (which do not have corner reflectors).

It is easy to see that distinct $2$-orbifolds may have the same
Euler-Satake characteristic even when they have homeomorphic
underlying spaces, as illustrated with the following.

\begin{example}
\label{ex-sameESCsameg}

Let $g \geq 0$ be an integer and $Q$ the orbifold with
underlying space $\Sigma_g$ and nine cone points, each of order $3$.
Let $Q^\prime$ be the orbifold with underlying space $\Sigma_g$ and
eight cone points, each of order $4$.  Then
\[
\begin{array}{rcl}
    \chi_{ES}(Q)
    &=&     -4 - 2g
                \\
    &=&     \chi_{ES}(Q^\prime).
\end{array}
\]

\end{example}

However, there can be only finitely many orbifolds with the same
Euler-Satake characteristic.

\begin{lemma}
\label{lem-finitesolutions}

Let $Q$ be a closed, connected, effective, orientable $2$-orbifold of genus $g$.
Then there are only finitely many closed, connected, effective, orientable
$2$-orbifolds with the same Euler-Satake characteristic.

\end{lemma}

\begin{proof}

We let $Q = \Sigma_g(m_1, \ldots , m_k)$ as above and
$m = m_k = \max\limits_{i = 1, \ldots,
k} m_i$.  Then
\[
\begin{array}{rcl}
    2 - 2g - k  + \frac{k}{m}
    &=&
    \frac{(2 - 2g - k)m + k}{m}       \\\\
    &\leq&
    \chi_{ES}(Q).
\end{array}
\]
Let $Q^\prime =\Sigma_{g^\prime}(m_1^\prime, \ldots, m_{k^\prime}^\prime)$
be another orbifold such that $\chi_{ES}(Q) = \chi_{ES}(Q^\prime)$.  Then as
each $m_i^\prime \geq 2$,
\begin{equation}
\label{eq-finiteQestimates}
\begin{array}{rcl}
    \frac{(2 - 2g - k)m + k}{m}
        &\leq&  \chi_{ES}(Q^\prime)
                    \\
        &=&     2 - 2g^\prime - k^\prime
                + \sum\limits_{i=1}^{k^\prime} \frac{1}{m_i^\prime}
                    \\
        &\leq&     2 - 2g^\prime - \frac{k^\prime}{2}
                    \\
        &\leq&  2 - 2g^\prime.
\end{array}
\end{equation}
It follows that
\[
    g^\prime    \leq    g + \frac{k(m - 1)}{2m},
\]
implying that there are only a finite number of possible values of
$g^\prime \geq 0$.  Using the estimate $[(2 - 2g - k)m + k]/m \leq 2 - 2g^\prime - k^\prime/2$
from Equation \ref{eq-finiteQestimates} above, it follows that
\[
    k^\prime
    \leq
    4g - 4g^\prime + \frac{2k(m-1)}{m},
\]
implying that for each possible value of $g^\prime$, there is a
finite number of possible values of $k^\prime \geq 0$.

To complete the proof, we fix values of $g^\prime$ and $k^\prime$
and show that the maximum isotropy order $m^\prime = \max\limits_{i
= 1, \ldots, k^\prime} m_i^\prime$ of $Q^\prime$ is bounded.  It
follows that there are a finite number of possible isotropy orders.
We have
\[
\begin{array}{rcl}
    \frac{(2 - 2g - k)m + k}{m}
    &\leq&
    \chi_{ES}(Q^\prime)                 \\
    &\leq&
    2 - 2g^\prime - k^\prime + \frac{k^\prime - 1}{2} +
    \frac{1}{m^\prime},
\end{array}
\]
implying that
\[
    m^\prime
    \leq
    \frac{2m}{m(1-4g+4g^\prime + k^\prime) - 2k(m-1)},
\]
completing the proof.

\end{proof}

This is no longer the case for the higher Euler-Satake
characteristics $\chi_{(l)}^{ES}$.  For instance,
\[
    \chi_{(1)}^{ES}(Q) = 2 - 2g
\]
coincides with the usual Euler characteristic of the underlying
space (note that this is the case in arbitrary dimension; see
\cite{farsiseaton3}).  It follows that this characteristic coincides
for any orbifolds with the same underlying space.  For $l > 1$,
infinite families of orbifolds whose $l$th Euler-Satake
characteristics coincide can be constructed.

\begin{example}
\label{ex-sameLESC}

Fix integers $j, l \geq 2$ with $j$ odd.  For each odd
integer $k \geq 1$, the orbifold $Q_k$ of genus $g_k =
\frac{1}{2}k(j^{l-1}-1)$ with $k$ cone points, each of order $j$,
satisfies
\[
\begin{array}{rcl}
    \chi_{(l)}^{ES}(Q_k)
        &=&     2 - 2g_k - k + j^{l-1}k     \\
        &=&     2.
\end{array}
\]

\end{example}

It is clear, then, that none of the $l$th Euler-Satake
characteristics classify this class of $2$-orbifolds.  However, as
stated in Theorem \ref{thrm-mainpositive}, the complete collection
of the $l$th Euler-Satake characteristics are sufficient for
classifying this class of orbifolds.  We have the following
technical result before proceeding to the proof of Theorem \ref{thrm-mainpositive}.

\begin{lemma}
\label{lem-removeequalconepoints}

Let $L$, be a nonnegative integer.  Suppose $Q$ and $Q^\prime$ are
closed, connected, effective, orientable $2$-orbifolds such that
\[
    \chi_{(l)}^{ES}(Q)
    =
    \chi_{(l)}^{ES}(Q^\prime)
\]
for $l \leq L$.  Suppose $Q$ and $Q^\prime$ both have at least one cone point of order
$m$.  If $\mathcal{Q}$ is the orbifold formed by removing a cone point of order $m$ from
$Q$ and $\mathcal{Q}^\prime$ the orbifold formed by removing a cone point of order $m$
from $Q^\prime$, then
\[
    \chi_{(l)}^{ES}(\mathcal{Q})
    =
    \chi_{(l)}^{ES}(\mathcal{Q}^\prime)
\]
for $l \leq L$.

\end{lemma}

\begin{proof}

We simply note that
\[
\begin{array}{rcl}
    \chi_{(l)}^{ES}(\mathcal{Q})
        &=&
        \chi_{(l)}^{ES}(Q) + 1 - m^{l-1}
                \\\\
        &=&
        \chi_{(l)}^{ES}(Q^\prime) + 1 - m^{l-1}
                \\\\
        &=&
        \chi_{(l)}^{ES}(\mathcal{Q}^\prime),
\end{array}
\]
for each $l \leq L$.

\end{proof}

\begin{proof}[Proof of Theorem \ref{thrm-mainpositive}]

Assume $Q$ and $Q^\prime$ are distinct, connected, effective, orientable $2$-orbifolds such that
$\chi_{(l)}^{ES} (Q) = \chi_{(l)}^{ES} (Q^\prime)$ for every nonnegative integer $l$.
Let $Q = \Sigma_g (m_1, \ldots, m_k)$
and $Q^\prime = \Sigma_{g^\prime}(m_1\prime,  \ldots, m_{k\prime}\prime)$ as above.
If $k = 0$ or $k^\prime = 0$, then the result is trivial, so assume not.
By Lemma \ref{lem-removeequalconepoints}, we can assume without loss of generality
that $m_k > m_{k^\prime}^\prime$.

Letting $l = 1$, we have that
$\chi_{(1)}^{ES} (Q) = 2 - 2g = \chi_{(1)}^{ES} (Q^\prime) = 2 - 2g^\prime$.
It follows that $g = g^\prime$.  Moreover, for each $l$, we have that
\[
    \left( \sum\limits_{i=1}^k m_i^{l-1} \right) - k
    =
    \left( \sum\limits_{i=1}^{k^\prime} (m_i^\prime)^{l-1} \right) - k^\prime .
\]
Noting that the left side is zero for at most one value of $l$, we have that for sufficiently large $l$,
\[
    \frac{\left( \sum\limits_{i=1}^k m_i^{l-1} \right) - k}
    {\left( \sum\limits_{i=1}^{k^\prime} (m_i^\prime)^{l-1} \right) - k^\prime} = 1.
\]
Based on the order relationships between the $m_i$ and $m_i^\prime$, we have
\[
\begin{array}{rcl}
    \frac{\left( \sum\limits_{i=1}^k m_i^{l-1} \right) - k}
    {\left( \sum\limits_{i=1}^{k^\prime} (m_i^\prime)^{l-1} \right) - k^\prime}
        &\geq&
    \frac{\left( \sum\limits_{i=1}^k m_i^{l-1} \right) - k}
    { \sum\limits_{i=1}^{k^\prime} (m_i^\prime)^{l-1}}
            \\\\
        &\geq&
    \frac{m_k^{l-1}  - k}
    {\sum\limits_{i=1}^{k^\prime} (m_i^\prime)^{l-1} }
            \\\\
        &\geq&
    \frac{m_k^{l-1}  - k}
    {k^\prime (m_{k^\prime}^\prime)^{l-1}}
            \\\\
        &=&
    \frac{m_k^{l-1}}{k^\prime (m_{k^\prime}^\prime)^{l-1}}
    - \frac{k}{k^\prime (m_{k^\prime}^\prime)^{l-1}}
            \\\\
        &=&
     \frac{1}{k^\prime} \left( \frac{m_k}{m_{k^\prime}^\prime} \right)^{l-1}
    - \frac{k}{k^\prime (m_{k^\prime}^\prime)^{l-1}}.
\end{array}
\]
However, as $m_k >  m_{k^\prime}^\prime$, it follows that
\[
    \lim\limits_{l \to \infty}
    \frac{1}{k^\prime} \left( \frac{m_k}{m_{k^\prime}^\prime} \right)^{l-1}
    - \frac{k}{k^\prime (m_{k^\prime}^\prime)^{l-1}}
    = \infty,
\]
a contradiction.  It follows that $Q = Q^\prime$.

\end{proof}


\subsection{Negative Classification Results for $\Gamma$ Free Abelian}
\label{subsec-negativeresults}

In this subsection, we demonstrate that Theorem
\ref{thrm-mainpositive} cannot be improved upon in the case of
closed, connected, effective, orientable $2$-orbifolds.  For any
finite collection of the $l$th Euler-Satake characteristics, we
construct an arbitrarily large (finite) collection of orbifolds
whose $l$th Euler-Satake characteristics coincide.  Specifically, the goal of this
section is to prove the following, which will be used to prove
Theorems \ref{thrm-mainnegative} and \ref{thrm-mainnegativegeneral}.
In particular, the perhaps mysterious conditions on the orders of
the cone points imposed throughout this section will allow us to
extend to the $\Gamma$-Euler-Satake characteristics for arbitrary
$\Gamma$ in Subsection \ref{subsec-generalgamma}.

\begin{proposition}
\label{prop-freenegative}

Let $L \geq 0$ and $N \geq 1$ be integers. Then there are
$N$ distinct closed, connected, effective, orientable
$2$-orbifolds $Q_0, Q_1, Q_2, \ldots , Q_N$ such that for
each $l = 0, 1, \ldots, L$,
\[
    \chi_{(l)}^{ES}(Q_0) = \chi_{(l)}^{ES}(Q_1) = \cdots = \chi_{(l)}^{ES}(Q_N).
\]
The common genus of these orbifolds can be taken to be any non-negative integer $g$.
Moreover, if $R$ is any collection of $2^{L - 2}$ integers $\geq 2$, then
the orders of the cone points of the $Q_j$ can be taken to be elements of the set
$\{ 2q + 1, 2q^2 + q, q + 2, 2q + q^2 : q \in R \}$.

\end{proposition}

First, we establish a number of results and constructions that will
simplify the arguments and notation in this section.

\begin{definition}
\label{def-operations}

Let $Q = \Sigma_g(m_1, \ldots, m_k)$ and $Q^\prime = \Sigma_g(m_1^\prime, \ldots, m_{k^\prime}^\prime)$
be two orbifolds with the same genus.  For any integer $s \geq 1$, we let
\[
    s \diamond Q
    =
    \Sigma_g(sm_1, \ldots , sm_k)
\]
denote the orbifold with the same genus and number of cone points as $Q$ such that the order of
each cone point is multiplied by $s$.  For any integer $t \geq 1$, we let
\[
    t \star Q
    =
    \Sigma_g \left(
        \stackrel{t}{\overbrace{m_1, \ldots, m_1}}, \;
        \stackrel{t}{\overbrace{m_2, \ldots, m_2}}, \; \ldots \;
        \stackrel{t}{\overbrace{m_k, \ldots, m_k}} \right)
\]
denote the orbifold with the same genus as $Q$ and each cone point of $Q$ occurring $t$ times.
We let
\[
    Q \circledast Q^\prime
    =
    \Sigma_g (m_1, \ldots, m_k, m_1^\prime, \ldots, m_{k^\prime}^\prime)
\]
denote the orbifold with the same genus as $Q$ and $Q^\prime$ and the combined $k + k^\prime$ cone points
of both $Q$ and $Q^\prime$.

\end{definition}

Note that $\circledast$ is clearly commutative and associative, and
\[
    t \star Q = \stackrel{t}{\overbrace{Q \circledast \ldots \circledast Q}}.
\]
Moreover, $1 \star Q = 1 \diamond Q = Q$.
In the case that the genus of $Q$ and $Q^\prime$ is zero, $Q \circledast Q^\prime$ corresponds to the connected sum (defined in the same way as manifolds with the additional assumption that the disks removed
contain no singular points) so that $t \star Q$ corresponds to the $t$-fold connected sum of $Q$
with itself.

\begin{lemma}
\label{lem-opersdontchangeeqs}

Let $L$, $s$, and $t$ be nonnegative integers.  Suppose $Q$ and $Q^\prime$ are
closed, connected, effective, orientable $2$-orbifolds with the same number of cone points such that
\[
    \chi_{(l)}^{ES}(Q)
    =
    \chi_{(l)}^{ES}(Q^\prime)
\]
for $l \leq L$.  Then
\[
    \chi_{(l)}^{ES}(t \star Q)
    =
    \chi_{(l)}^{ES}(t \star Q^\prime)
\]
and
\[
    \chi_{(l)}^{ES}(s \diamond Q)
    =
    \chi_{(l)}^{ES}(s \diamond Q^\prime)
\]
for each $l \leq L$.

\end{lemma}

\begin{proof}

Assume $Q = \Sigma_g(m_1, \ldots, m_k)$ and $Q^\prime = \Sigma_g(m_1^\prime, \ldots, m_k^\prime)$.
The result then follows from direct computations and application of Proposition
\ref{prop-orientedECformula}.

\end{proof}

\begin{lemma}
\label{lem-makeconepointssame}

Let $L$ be a nonnegative integer, and let
$Q_1, Q_1^\prime, \ldots , Q_N, Q_N^\prime$ be closed, connected, effective
$2$-orbifolds.  Assume that for each $j = 1, \ldots N$, $Q_j$ and $Q_j^\prime$ have the
same number of cone points, and
\[
    \chi_{(l)}^{ES}(Q_j)
    =
    \chi_{(l)}^{ES}(Q_j^\prime)
\]
for each $l \leq L$.  Then there are closed, connected, effective, orientable $2$-orbifolds
$\mathcal{Q}_1, \mathcal{Q}_1^\prime, \ldots , \mathcal{Q}_N, \mathcal{Q}_N^\prime$
all with the same number of cone points
such that for each $j = 1, \ldots N$,
\[
    \mathcal{Q}_j = t_j\star Q_j
\]
and
\[
    \mathcal{Q}_j^\prime = t_j\star Q_j^\prime
\]
for integers $t_j \geq 1$, and
\[
    \chi_{(l)}^{ES}(\mathcal{Q}_j)
    =
    \chi_{(l)}^{ES}(\mathcal{Q}_j^\prime)
\]
for each $l \leq L$.

\end{lemma}

\begin{proof}

For each $j = 1, \ldots N$, let $k_j$ be the common number of cone points of $Q_j$ and
$Q_j^\prime$.  Then set
\[
    t_j = \prod\limits_{i = 1, i\neq j}^N k_i ,
\]
\[
    \mathcal{Q}_j = t_j \star Q_j,
\]
and
\[
    \mathcal{Q}_j^\prime = t_j \star Q_j^\prime.
\]
By Lemma \ref{lem-opersdontchangeeqs},
\[
    \chi_{(l)}^{ES}(\mathcal{Q}_j)
    =
    \chi_{(l)}^{ES}(\mathcal{Q}_j^\prime)
\]
for each $l \leq L$.  Moreover, each $\mathcal{Q}_j$ and
$\mathcal{Q}_j^\prime$ has $t_j k_j = \prod_{i=1}^N k_i$ cone points.

\end{proof}

In the following lemma, we establish an infinite family of pairs
of orbifolds with the same $l$th Euler-Satake characteristic for $l = 0, 1, 2$
and a number of other properties, each of which being required for constructions
in the sequel.

\begin{lemma}
\label{lem-basecase}

For each integer $q \geq 2$ and each $g \geq 0$, let
\[
    Q[g, q] = \Sigma_g(2q + 1, 2q + 1, 2q^2 + q)
\]
and
\[
    Q^\prime[g, q] = \Sigma_g(q + 2, q^2 + 2q, q^2 + 2q).
\]
Then
\[
    \chi_{(l)}^{ES}(Q[g,q])
    =
    \chi_{(l)}^{ES}(Q^\prime[g,q])
\]
for $l = 0, 1, 2$.  The orbifolds $\{ Q[g, q], Q^\prime[g,q] : g \geq 0, q \geq 2 \}$
are all distinct.   Moreover, if
$Q$ and $s\diamond Q$ are elements of $\{ Q[g, q], Q^\prime[g,q] : g \geq 0, q \geq 2 \}$
for some orbifold $Q$ and integer $s$, then $s = 1$.

\end{lemma}

\begin{proof}

Applying Proposition \ref{prop-orientedECformula}, for each integer $q \geq 2$ we have
\[
\begin{array}{rcl}
    \chi_{(0)}^{ES}(Q[g,q])
        &=&     \frac{1}{q} -1 - 2g
            \\
        &=&     \chi_{(0)}^{ES}(Q^\prime[g,q]),
\end{array}
\]
\[
\begin{array}{rcl}
    \chi_{(1)}^{ES}(Q[g,q])
        &=&     2 - 2g
                \\
        &=&     \chi_{(1)}^{ES}(Q^\prime[g,q]),
\end{array}
\]
and
\[
\begin{array}{rcl}
    \chi_{(2)}^{ES}(Q[g,q])
        &=&     1 - 2g + 5q + 2q^2
                \\
        &=&     \chi_{(2)}^{ES}(Q^\prime[g,q]).
\end{array}
\]

That these orbifolds are all distinct is obvious; it is impossible that
$Q[g,r] = Q^\prime[g,q]$, as $Q[g,r]$ has two smaller and one larger
order cone point while $Q^\prime[g,q]$ has one smaller and two larger.
Moreover, $Q[g,r] = Q[g,q]$ implies that $2r + 1 = 2q + 1$ so that $r = q$,
and similarly $Q^\prime[g,r] = Q^\prime[g,q]$ implies that $r + 2 = q +2$
so that $r = q$.  The remaining claim is clear.

\end{proof}

\begin{lemma}
\label{lem-inductivestep}

For each nonnegative integer $L$ and any genus $g$, there is a pair of distinct,
closed, connected, effective, orientable $2$-orbifolds $Q$ and $Q^\prime$ with the same number of cone points
such that $\chi_{(l)}^{ES}(Q) = \chi_{(l)}^{ES}(Q^\prime)$ for each $l
\leq L$.  The common genus of $Q$ and $Q^\prime$ can be taken to be any non-negative integer $g$.
Moreover, if $R$ is any collection of $2^{L - 2}$ integers $\geq 2$, then
the orders of the cone points of $Q$ and $Q^\prime$ can be taken to be elements of the set
$\{ 2q + 1, 2q^2 + q, q + 2, 2q + q^2 : q \in R \}$.

\end{lemma}

\begin{proof}

Throughout, we assume all orbifolds have a fixed genus $g$; note
that the constructions in this proof hold for any value of $g$.

Let $L \geq 3$ be an integer, and let $q : \{ 1, 2, \ldots, 2^{L-2} \} \rightarrow \{ 2, 3, \ldots \}$
be the order-preserving function whose image is $R$;
that is, $q(j_1) < q(j_2)$ whenever $j_1 < j_2$.
For $j = 1, \ldots 2^{L - 2}$, let $Q_{j,2} = Q[g,q(j)]$ and $Q_{j,2}^\prime = Q^\prime[g,q(j)]$ be the orbifolds constructed in Lemma \ref{lem-basecase}.  Here, the subscript
$2$ indicates that the $Q_{j,2}$ and $Q_{j,2}^\prime$ have the same $l$th Euler-Satake characteristic for
$l \leq 2$.  To summarize what follows, we construct from these $2^{L-2}$ pairs of orbifolds whose
$l$th Euler-Satake characteristics coincide for $l \leq 2$ a collection of $2^{L-3}$ pairs of orbifolds
whose $l$th Euler-Satake characteristics coincide for $l \leq 3$.
Continuing recursively, we construct a pair of orbifolds $Q = Q_{1,L}$ and $Q^\prime = Q_{1,L}^\prime$ whose $l$th
Euler-Satake characteristics coincide for $l \leq L$.

The following describes the recursive step in detail.
Let $n \geq 3$ and $1 \leq j \leq 2^{L - n}$ with $j$ odd, and assume that
there are orbifolds
$Q_{j,n} = \Sigma_g (a_1, a_2, \ldots , a_k)$,
$Q_{j,n}^\prime = \Sigma_g(b_1, b_2, \ldots , b_k)$,
$Q_{j+1,n} = \Sigma_g(c_1, c_2, \ldots , c_k)$, and
$Q_{j+1,n}^\prime = \Sigma_g(d_1, d_2, \ldots, d_k)$ with
$a_1, a_2, \ldots, a_k$, $b_1, b_2, \ldots, b_k$, $c_1, c_2, \ldots, c_k$,
$d_1, d_2, \ldots, d_k \geq 2$
integers such that
\[
    \chi_{(l)}^{ES}(Q_{j,n})
    =
    \chi_{(l)}^{ES}(Q_{j,n}^\prime)
\]
and
\[
    \chi_{(l)}^{ES}(Q_{j+1,n})
    =
    \chi_{(l)}^{ES}(Q_{j+1,n}^\prime)
\]
for each $l = 0, 1, 2, \ldots , n$.  Note that this implies that
\begin{equation}
\label{eq-inductionsumscoincide1}
    \sum\limits_{i=1}^k a_i^{l-1}
    =
    \sum\limits_{i=1}^k b_i^{l-1}
\end{equation}
and
\begin{equation}
\label{eq-inductionsumscoincide2}
    \sum\limits_{i=1}^k c_i^{l-1}
    =
    \sum\limits_{i=1}^k d_i^{l-1}
\end{equation}
for each $l = 0, 1, 2, \ldots , n$.

If $\chi_{(n+1)}^{ES}(Q_{j,n}) =\chi_{(n)}^{ES}(Q_{j,n}^\prime)$, then set $Q_{(j+1)/2  ,n+1} = Q_{j,n}$,
$Q_{(j+1)/2,n+1}^\prime = Q_{j,n}^\prime$.
Similarly, if $\chi_{(n+1)}^{ES}(Q_{j+1,n}) =\chi_{(n)}^{ES}(Q_{j+1,n}^\prime)$, then set
$Q_{(j+1)/2,n+1} = Q_{j+1,n}$ and $Q_{(j+1)/2,n+1}^\prime = Q_{j+1,n}^\prime$.  Otherwise, define
\[
    \delta_1 =
    \sum\limits_{i=1}^k a_i^n
    -
    \sum\limits_{i=1}^k b_i^n
\]
and
\[
    \delta_2 =
    \sum\limits_{i=1}^k d_i^n
    -
    \sum\limits_{i=1}^k c_i^n.
\]
Note that if $\delta_1 = 0$ then $\chi_{(n+1)}^{ES}(Q_{j,n}) = \chi_{(n+1)}^{ES}(Q_{j,n}^\prime)$,
so we can assume by switching the roles of $Q_{j,n}$ and $Q_{j,n}^\prime$ if necessary that $\delta_1 >0$.
Similarly, we assume with no loss of generality that $\delta_2 > 0$.

We construct the orbifolds $Q_{(j+1)/2,n+1}$ and $Q_{(j+1)/2,n+1}^\prime$ as follows.  Let
\[
    Q_{(j+1)/2,n+1}   =
    (\delta_2 \star Q_{j,n}) \circledast (\delta_1 \star Q_{j+1,n})
\]
and
\[
    Q_{(j+1)/2,n+1}^\prime   =
    (\delta_2 \star Q_{j,n}^\prime) \circledast (\delta_1 \star Q_{j+1,n}^\prime).
\]
That is,
\[
    Q_{(j+1)/2,n+1} = \Sigma_g\left(
            \stackrel{\delta_2}{\overbrace{a_1, \ldots a_1}}, \;
            \stackrel{\delta_2}{\overbrace{a_2, \ldots a_2}},  \; \ldots \;
            \stackrel{\delta_2}{\overbrace{a_k, \ldots a_k}}, \;
            \stackrel{\delta_1}{\overbrace{c_1, \ldots c_1}},\;
            \stackrel{\delta_1}{\overbrace{c_2, \ldots c_2}}, \; \ldots \;
            \stackrel{\delta_1}{\overbrace{c_k, \ldots c_k}}
            \right),
\]
and
\[
    Q_{(j+1)/2,n+1}^\prime = \Sigma_g\left(
            \stackrel{\delta_2}{\overbrace{b_1, \ldots b_1}}, \;
            \stackrel{\delta_2}{\overbrace{b_2, \ldots b_2}},  \; \ldots \;
            \stackrel{\delta_2}{\overbrace{b_k, \ldots b_k}}, \;
            \stackrel{\delta_1}{\overbrace{d_1, \ldots d_1}},\;
            \stackrel{\delta_1}{\overbrace{d_2, \ldots d_2}}, \; \ldots \;
            \stackrel{\delta_1}{\overbrace{d_k, \ldots d_k}}\right).
\]
Then
\[
\begin{array}{rcl}
    \chi_{(n+1)}^{ES}(Q_{(j+1)/2,n+1}) - \chi_{(n+1)}^{ES}(Q_{(j+1)/2,n+1}^\prime)
&=&
    \left(
    2 - 2g  - (\delta_2 k + \delta_1 k)
        + \delta_2 \sum\limits_{i=1}^k a_i^n
        + \delta_1 \sum\limits_{i=1}^k c_i^n
    \right)
    \\
    && -
    \left(
    2 - 2g  - (\delta_2 k + \delta_1 k)
        + \delta_2 \sum\limits_{i=1}^k b_i^n
        + \delta_1 \sum\limits_{i=1}^k d_i^n
    \right)
                \\\\
&=&
    \delta_2 \left(
    \sum\limits_{i=1}^{k_1} a_i^n - \sum\limits_{i=1}^{k_2} b_i^n
    \right)
    + \delta_1 \left(
    \sum\limits_{i=1}^{k_4} c_i^n - \sum\limits_{i=1}^{k_3} d_i^n
    \right)
                \\\\
&=&
    \delta_2 \delta_1 + \delta_1 (-\delta_2)
                \\
&=&
    0,
\end{array}
\]
so that
\[
    \chi_{(n+1)}^{ES}(Q_{(j+1)/2,n+1})
    =
    \chi_{(n+1)}^{ES}(Q_{(j+1)/2,n+1}^\prime).
\]
Moreover, for each nonnegative integer $l \leq n$,
\[
\begin{array}{rcl}
    \chi_{(l)}^{ES}(Q_{(j+1)/2,n+1}) - \chi_{(l)}^{ES}(Q_{(j+1)/2,n+1}^\prime)
&=&
    \left(
    2 - 2g  - (\delta_2 k + \delta_1 k)
        + \delta_2 \sum\limits_{i=1}^k a_i^{l-1}
        + \delta_1 \sum\limits_{i=1}^k c_i^{l-1}
    \right)
    \\
    && -
    \left(
    2 - 2g  - (\delta_2 k + \delta_1 k)
        + \delta_2 \sum\limits_{i=1}^k b_i^{l - 1}
        + \delta_1 \sum\limits_{i=1}^k d_i^{l - 1}
    \right)
                \\\\
&=&
    \delta_2 \left( \sum\limits_{i=1}^k a_i^{l-1} -
    \sum\limits_{i=1}^k b_i^{l - 1} \right)
    - \delta_1 \left( \sum\limits_{i=1}^k c_i^{l-1} -
    \sum\limits_{i=1}^k d_i^{l - 1} \right)
                \\
&=& 0
\end{array}
\]
by Equations \ref{eq-inductionsumscoincide1} and \ref{eq-inductionsumscoincide2}.

For each $n \geq 3$, we apply this construction for each odd $j$ with $1 \leq j \leq 2^{L-n}$,
forming $2^{L-n-1}$ orbifold pairs $Q_{(j+1)/2,n+1}$ and $Q_{(j+1)/2,n+1}^\prime$; note that these are
indexed as $Q_{r, n+1}$, $Q_{r, n+1}^\prime$ for $r = 1, 2, \ldots , 2^{L-n-1}$.  For each $r$,
$Q_{r, n+1}$ and $Q_{r, n+1}^\prime$ have the same number of cone points; hence, we can apply Lemma \ref{lem-makeconepointssame} to the collection of
$\{ Q_{r, n+1}, Q_{r, n+1}^\prime : r=1,\ldots,2^{L-n-1} \}$ to assume that they
all have the same number of cone
points, which is required in the next recursive step.  The result is a pair of orbifolds
$Q = Q_{1,L}$ and $Q^\prime = Q_{1,L}^\prime$ with the desired properties.  It remains only to show
that $Q$ and $Q^\prime$ are distinct.

While not all of the $Q_{j, 2}$ and $Q_{j, 2}^\prime$ may have been used in this construction
(if it happens that
$\chi_{(n+1)}^{ES}(Q_{j,n}) =\chi_{(n)}^{ES}(Q_{j,n}^\prime)$ or
$\chi_{(n+1)}^{ES}(Q_{j+1,n}) =\chi_{(n)}^{ES}(Q_{j+1,n}^\prime)$
for some $j$), but note that both $Q$ and $Q^\prime$ have at least three cone points each.
Fix the smallest value of $j$ such that $Q$ and $Q^\prime$ have cone points arising from $Q[g,q(j)]$ and $Q^\prime[g,q(j)]$.  While the roles of these two may have switched to ensure that $\delta_1$ and
$\delta_2$ are positive, only one of $Q$ and $Q^\prime$ can have cone points of order
$q(j) + 2$ from $Q^\prime[q(j)]$.  As $q(j) \geq 2$ for all $j$ and $q(j)$ is strictly increasing, it
follows that all other cone points of the two orbifolds must be strictly greater than $q(j) + 2$,
and hence that $Q$ and $Q^\prime$ are distinct orbifolds.

\end{proof}

\begin{lemma}
\label{lem-constructaddsolcmult}

Let $L$, be a nonnegative integer.  Suppose $Q$ and $Q^\prime$ are distinct,
closed, connected, effective, orientable $2$-orbifolds with the same genus and same number of cone points such that
\[
    \chi_{(l)}^{ES}(Q)
    =
    \chi_{(l)}^{ES}(Q^\prime)
\]
for $l \leq L$.  For any integer $N \geq 2$, there is a collection
$\mathcal{Q}_1, \mathcal{Q}_2, \ldots \mathcal{Q}_N$ of distinct
closed, effective, orientable $2$-orbifolds such that
\[
    \chi_{(l)}^{ES}(\mathcal{Q}_1)
    =
    \chi_{(l)}^{ES}(\mathcal{Q}_2)
    = \cdots =
    \chi_{(l)}^{ES}(\mathcal{Q}_N)
\]
for $l \leq L$.  Moreover, the orders of cone points of each
$\mathcal{Q}_j$ are those of $Q$ and $Q^\prime$ only.

\end{lemma}

\begin{proof}

It is obvious that $Q$ and $Q^\prime$ must have singular points, as otherwise
$\chi_{(0)}^{ES}(Q)=\chi_{(0)}^{ES}(Q^\prime)$ implies that $Q = Q^\prime$.
By this observation and Lemma
\ref{lem-removeequalconepoints}, we may assume without
loss of generality that $Q$ has $r$ cone points of order $m$ for some $m \geq 2$,
and $Q^\prime$ does not have a cone point of order $m$.  Let $k$ be the common number
of cone points of $Q$ and $Q^\prime$.

For $j = 1, 2, \ldots, N$, we define
\[
    \mathcal{Q}_j = \stackrel{N-j}{\overbrace{Q \circledast \cdots \circledast Q}}
        \;\circledast\;
        \stackrel{j - 1}{\overbrace{Q^\prime \circledast \cdots \circledast Q^\prime}}.
\]
Then each $\mathcal{Q}_j$ has exactly $(N-j)r$ cone points of order $m$ so that the $\mathcal{Q}_j$ are distinct.

Now, let $Q = \Sigma_g(a_1, \ldots , a_k)$ and
$Q^\prime = \Sigma_g(b_1, \ldots , b_k)$ (so that in particular, $a_i = m$ for $r$
choices or $i$), and note that $\sum_{i=1}^k a_i^{l-1} = \sum_{i=1}^k b_i^{l-1}$
for each $l \leq L$.  For each $0 \leq l \leq L$ and $1 \leq j \leq N$, we compute
\[
\begin{array}{rcl}
    \chi_{(l)}^{ES}(\mathcal{Q}_j)
        &=&
        2 - 2g - (N-1)k + (N - j)\sum\limits_{i=1}^k a_i^{l-1}
        + (j - 1)\sum\limits_{i=1}^k b_i^{l-1}
                    \\
        &=&
        2 - 2g - (N-1)k + (N - j)\sum\limits_{i=1}^k a_i^{l-1}
        + (j - 1)\sum\limits_{i=1}^k a_i^{l-1}
                    \\
        &=&
        2 - 2g - (N-1)k + (N - 1)\sum\limits_{i=1}^k a_i^{l-1}.
\end{array}
\]
As $\chi_{(l)}^{ES}(\mathcal{Q}_j)$ does not depend on $j$, we are done.

\end{proof}

\begin{proof}[Proof of Proposition \ref{prop-freenegative}]

By Lemmas \ref{lem-basecase} and \ref{lem-inductivestep}, there
exists a pair of orbifolds with the desired properties.  By Lemma
\ref{lem-constructaddsolcmult}, there are $N$ such orbifolds.

\end{proof}


\subsection{Negative Classification Results for General $\Gamma$}
\label{subsec-generalgamma}

Let $\mathfrak{G}$ be a set of
finitely generated discrete groups, and let $\mathfrak{A} = \{ \Gamma
/[\Gamma, \Gamma]: \Gamma \in \mathfrak{G} \}$ denote the collection
of abelianizations of elements of $\mathfrak{G}$.
Then each $\Gamma /[\Gamma,\Gamma]$ is of the form $\Z^l \oplus G$ uniquely for
$l \geq 0$ and $G$ finite by the Fundamental Theorem of Finitely
Generated Abelian Groups.  Let $\mathfrak{F}$ denote the set of $G$
that appear in this decomposition for elements of $\mathfrak{A}$; that is
\[
    \mathfrak{F} = \{ G : \Z^l \oplus G \in \mathfrak{A} \}.
\]
Let $\mathfrak{P}$ denote the set of primes $p$ such that there is a
$G \in \mathfrak{F}$ and $g \in G$ with $|g|$ divisible by $p$.
In this section, we prove the following.

\begin{theorem}
\label{thrm-mainnegativegeneral}

Let $N \geq 2$ be an integer.
Let $\mathfrak{G}$ be a nonempty set of finitely
generated discrete groups such that the ranks of the elements of
$\mathfrak{A}$ are bounded, and $\mathfrak{P}$ is finite.
Then there are distinct, closed, connected, effective, orientable
$2$-orbifolds $Q_1, Q_2, \ldots , Q_N$ such that for each $\Gamma \in
\mathfrak{G}$,
\[
    \chi_\Gamma^{ES}(Q_1) = \chi_\Gamma^{ES}(Q_2) = \cdots = \chi_\Gamma^{ES}(Q_N).
\]
The common genus of the $Q_j$ can be chosen to be any nonnegative integer.

\end{theorem}

In particular, note that the hypotheses of Theorem \ref{thrm-mainnegativegeneral}
are obviously satisfied for $\mathfrak{G}$ finite; hence Theorem \ref{thrm-mainnegative}
is a trivial consequence.  First, we have the following.  Recall that an \emph{abelian
orbifold} is an orbifold $Q$ such that every isotropy group of $Q$ is abelian.

\begin{lemma}
\label{lem-fingenmodcommutators}

Let $Q$ be an abelian orbifold and $\Gamma$ a finitely generated
discrete group.  Then $\tilde{Q}_\Gamma$ and $\tilde{Q}_{\Gamma/[\Gamma, \Gamma]}$
are diffeomorphic.  In particular, if $Q$ is closed, then
\[
    \chi_\Gamma^{ES}(Q)    =   \chi_{\Gamma/[\Gamma,\Gamma]}^{ES}(Q) .
\]

\end{lemma}

\begin{proof}

Let $\rho :  \Gamma \rightarrow \Gamma/[\Gamma,\Gamma]$
denote the quotient map.  For each local group $G_x$ of $Q$, it is easy to see that as $G_x$ is abelian,
the correspondence $\phi_x \mapsto \phi_x \circ \rho$ is a bijection between $\mbox{HOM}(\Gamma, G_x)$ and
$\mbox{HOM}(\Gamma/[\Gamma,\Gamma], G_x)$.  It clearly follows that
\[
\begin{array}{rccl}
    e_\rho :&\mathcal{S}_\mathcal{G}^\Gamma
            &\longrightarrow&
            \mathcal{S}_\mathcal{G}^{\Gamma/[\Gamma, \Gamma]}        \\
    :&\phi_x              &\longmapsto        &\phi_x \circ \rho
\end{array}
\]
is a bijective.  See \cite[Section 3.3]{farsiseaton3} for a more general treatment of
maps on sectors induced by group homomorphisms, of which $e_\rho$ is an example.

Recall that if $\{V_x, G_x, \pi_x \}$ is a linear chart for $\mathcal{G}$ at $x$, then
$\{ V_x^{\langle \phi_x \rangle}, C_{G_x}(\phi_x), \pi_x^{\phi_x} \}$ is a linear chart for
$\mathcal{S}_\mathcal{G}^\Gamma$ at $\phi_x$.   As, $\Imt \phi_x = \Imt \phi_x \circ \rho \leq G_x$,
it follows that $e_\rho$ is simply the identity on charts and hence a  $\mathcal{G}$-equivariant
diffeomorphism.  It hence induces an isomorphism of orbifold groupoids between $\mathcal{G}^\Gamma$ and
$\mathcal{G}^{\Gamma/[\Gamma, \Gamma]}$.

\end{proof}

It follows that, for abelian orbifolds
$Q$ and $Q^\prime$,
\[
    \chi_\Gamma^{ES}(Q)    =   \chi_\Gamma^{ES}(Q^\prime) \;\;\; \forall \; \Gamma \in \mathfrak{G}
\]
if and only if
\[
    \chi_\Lambda^{ES}(Q)    =   \chi_\Lambda^{ES}(Q^\prime) \;\;\; \forall \; \Lambda \in \mathfrak{A}.
\]

\begin{proof}[Proof of Theorem \ref{thrm-mainnegativegeneral}]

Suppose $L$ is the maximum rank of the elements of $\mathfrak{A}$.
If $\mathfrak{P}$ is empty, then $\mathfrak{A}$ contains only free abelian
groups, and the result follows from Proposition \ref{prop-freenegative}.  So assume
$\mathfrak{P}\neq\emptyset$.

Let $\mathfrak{P} = \{ p_1, p_2, \ldots , p_r \}$, and define
\[
\begin{array}{rccl}
    q:  & \{ 1, 2, \ldots , 2^{L-1} \}
        &\longrightarrow&   \{ 2, 3, \ldots \}          \\
    :   & j
        &\longmapsto&   j\left( 2\prod\limits_{i=1}^r p_i \right) - 1.
\end{array}
\]
Then $q$ is order-preserving, and $q(j) \geq 2$ for each $j$.  Moreover,
for each $i$ and $j$, $q(j) \equiv -1$ mod $p_i$. Hence $2q(j) + 1 \equiv -1$ mod
$p_i$, and $q(j) + 2 \equiv 1$ mod $p_i$.  It follows that $q(j)$, $2q(j) +
1$, and $q(j)+ 2$ are not divisible by any element of $\mathcal{P}$.

By Proposition \ref{prop-freenegative}, for any choice of genus,
there are orbifolds $Q_1, \ldots , Q_N$ such that
$\chi_{(l)}^{ES}(Q_1) = \chi_{(l)}^{ES}(Q_2) = \cdots =
\chi_{(l)}^{ES}(Q_N)$ for each $l \leq L$. Moreover, we can choose
the $Q_j$ so that their cone points all have orders $q(j)+1, 2q(j)^2
+ q(j), q(j) + 2, 2q(j) + q(j)^2$ for values of $j \in \{ 1, 2,
\ldots, 2^{L-1} \}$; in particular, we use the function $q$ as
defined above in the proof of Lemma \ref{lem-inductivestep}.

Fix some $j$ and let $q = q(j)$.  For any homomorphism $\phi : G \rightarrow \Z/(2q + 1)\Z$ with $G$
finite, for each $g \in G$, $|g|$ must be divisible by the order of
$|\phi(g)|$, which must divide $2q + 1$. However, $|g|$ and $2q + 1$
are relatively prime by construction so that $|\phi(g)| = 1$ and $g
\in \mbox{Ker}\; \phi$. Hence, $\phi$ is the trivial homomorphism.
The same argument applies to homomorphisms
into $\Z/(2q^2 + q)\Z$, $\Z/(2q + q^2)\Z$, and $\Z/(q + 2)\Z$.

It follows that for any homomorphism  $\phi : \Z^l \oplus G
\rightarrow \Z/m\Z$ where $m = 2q + 1, 2q^2 + q, 2q + q^2$, or $q +
2$, each $g\in G$ is in the kernel, so that
\[
    \chi_{\Z^l \oplus G}^{ES}(Q_j)
        =     \chi_{(l)}^{ES}(Q_j)
\]
for each $\Z^l \oplus G \in \mathfrak{A}$ and each $j$, completing the proof.

\end{proof}


\section{Other Classes of Orbifolds}
\label{sec-otherorbs}

In this section, we demonstrate that the hypotheses of Theorem
\ref{thrm-mainpositive} cannot be relaxed to include
noneffective nor non-orientable orbifolds.  Note that in
the case of a global quotient, it is convenient to
describe the $\Gamma$-sectors globally as originally given in \cite{tamanoi1}.
See Equation \ref{eq-globalsectordescrip} above and \cite[Section 3.1]{farsiseaton2} for the equivalence of
these definitions.

\begin{example}
\label{ex-noneffective}

Let $\Z/6\Z = \langle a \rangle$ act on $S^2$ so that $a$ acts as a
rotation through $\pi/3$, and let $Q$ denote the resulting quotient orbifold. Then
$Q$ is effective, has underlying space homeomorphic to $S^2$, and
has two cone points, both with isotropy $\Z/6\Z$.  Similarly, let
$\Z/6\Z = \langle b \rangle$ act on $S^2$ where $b$ acts by a
rotation through $2\pi/3$. Then the quotient orbifold $Q^\prime$ has two cone
points with isotropy $\Z/6\Z$, and every other point has isotropy
$\Z/3\Z$.  Let $np, sp \in S^2$ denote the two fixed points of each
of these actions.  We claim that $\chi_\Gamma^{ES}(Q) =
\chi_\Gamma^{ES}(Q^\prime)$ for every finitely generated discrete $\Gamma$.

Let $\iota : a \mapsto b$ denote the obvious
isomorphism and fix $\Gamma$ finitely generated and discrete.  Then $\phi \mapsto
\iota \circ \phi$ of course defines a bijection between
$\mbox{HOM}(\Gamma, \langle a \rangle)$ and $\mbox{HOM}(\Gamma,
\langle b \rangle)$.  We note the following.

\begin{itemize}

\item   If $\Imt \phi = \langle 1\rangle$, then $(S^2;\langle a \rangle)_{(\phi)}
= S^2\rtimes \langle a\rangle$, diffeomorphic to $Q$, has Euler-Satake characteristic
$\frac{1}{3}$ and $(S^2;\langle b \rangle)_{(\iota \circ \phi)} =
S^2 \rtimes \langle b \rangle$, diffeomorphic to $Q^\prime$, has Euler characteristic
$\frac{1}{3}$.

\item   If $\Imt \phi = \langle a\rangle$ or $\langle a^2\rangle$,
then $(S^2;\langle a \rangle)_{(\phi)} = \{ np, sp \} \rtimes \langle a\rangle$
has Euler-Satake characteristic
$\frac{1}{3}$ and $(S^2;\langle a \rangle)_{(\iota \circ \phi)} = \{ np,
sp \} \rtimes \langle b \rangle$ has Euler characteristic
$\frac{1}{3}$.

\item   If $\Imt \phi = \langle a^3\rangle$, then
$(S^2;\langle a \rangle)_{(\phi)} = \{ np, sp \} \rtimes \langle a\rangle$ has
Euler-Satake characteristic $\frac{1}{3}$ and
$(S^2;\langle a \rangle)_{(\iota \circ \phi)} = S^2 \rtimes \langle b \rangle$,
diffeomorphic to $Q^\prime$ has Euler characteristic $\frac{1}{3}$.

\end{itemize}

It follows that
\[
    \chi_\Gamma^{ES}\left(\widetilde{(Q)}_{(\phi)}\right)
    =
    \chi_\Gamma^{ES}\left(\widetilde{(Q^\prime)}_{(\iota\circ\phi)}\right)
\]
for each $\phi \in \mbox{HOM}(\Gamma, \langle a \rangle)$.  Hence
there is no finitely generated discrete $\Gamma$ such that $\chi_\Gamma^{ES}$
distinguishes between $Q$ and $Q^\prime$.

\end{example}

\begin{example}
\label{ex-nonorientable}

Let $Q$ and $Q^\prime$
be the orbifolds homeomorphic as topological spaces to the
cylinder $S^1 \times [0, 1]$.
Let $B_0 = S^1 \times \{ 0 \}$ and
$B_1 = S^1 \times \{ 1 \}$ denote the boundary components of
$Q$, and similarly $B_0^\prime$ and $B_1^\prime$ the boundary components
of $Q^\prime$.  Both orbifolds have four corner reflectors as follows,
where $D_{2n}$ denotes the dihedral group of order $2n$.
The orbifold $Q$ has corner reflectors modeled by
$\R^2/D_6$ and $\R^2/D_{10}$ on $B_0$; and
$\R^2/D_{14}$ and $\R^2/D_{22}$ on $B_1$.
The orbifold $Q^\prime$ has corner reflectors modeled by
$\R^2/D_6$ and $\R^2/D_{14}$ on $B_0^\prime$; and
$\R^2/D_{10}$ and $\R^2/D_{22}$ on $B_1^\prime$.  By examining
boundary components, it is clear that $Q$ and $Q^\prime$ are
not diffeomorphic.

As all dihedral groups under consideration have an odd number of rotations
and hence the centralizer of an element of order $2$ is precisely the
group generated by that element, it is easy to see that
the $\Gamma$-sectors of $Q$
for each finitely generated discrete group $\Gamma$
all occur in the following list:
\begin{itemize}
\item   an orbifold diffeomorphic to $Q$,
\item   a circle with trivial $\Z/2\Z$-action, and
\item   a point with trivial $\Z/n\Z$-action, where $n = 3, 5, 7$, or $11$.
\end{itemize}
Similarly, the $\Gamma$-sectors of $Q^\prime$ are of the form
\begin{itemize}
\item   an orbifold diffeomorphic to $Q^\prime$,
\item   a circle with trivial $\Z/2\Z$-action, and
\item   a point with trivial $\Z/n\Z$-action, where $n = 3, 5, 7$, or $11$.
\end{itemize}
There is an obvious bijection between homomorphisms
from $\Gamma$ into the local groups of $Q$ and
homomorphisms from $\Gamma$ into the local groups of $Q^\prime$.
This bijection preserves the diffeomorphism class of the
corresponding sector in every case except that of the trivial homomorphism,
corresponding to the unique sectors diffeomorphic to $Q$ and $Q^\prime$.
However, as
\[
\begin{array}{rcl}
    \chi_{ES}(Q)
    &=&     \chi_{ES}(Q^\prime)
    \\\\
    &=&     -2 + \frac{1}{6} + \frac{1}{10} + \frac{1}{14} + \frac{1}{22}
    \\\\
    &=&
    \frac{-1867}{1155},
\end{array}
\]
it follows that $\chi_\Gamma^{ES}(Q) = \chi_\Gamma^{ES}(Q^\prime)$
for every finitely generated discrete $\Gamma$.

\end{example}

Finally, we note that constructions of orbifolds whose
$\Gamma$-Euler-Satake characteristics coincide can be used to
construct orbifolds of arbitrary even dimension with the same
properties.

\begin{corollary}
\label{cor-generaldim}

Let $N, n \geq 2$ be integers with $n$ even.  Let
$\mathfrak{G}$ be a nonempty collection of finitely generated
discrete groups such that, with the notation as in Subsection
\ref{subsec-generalgamma}, the ranks of the elements of
$\mathfrak{A}$ are bounded, and $\mathfrak{P}$ is finite. Then there
are distinct closed, connected, effective, orientable
$n$-dimensional orbifolds $Q_1, Q_2, \ldots , Q_N$ such that for
each $\Gamma \in \mathfrak{G}$,
\[
    \chi_\Gamma^{ES}(Q_1) = \chi_\Gamma^{ES}(Q_2) = \cdots = \chi_\Gamma^{ES}(Q_N).
\]

\end{corollary}

\begin{proof}

Since the $\Gamma$-Euler-Satake characteristic is multiplicative
(see \cite[Section 4.1]{farsiseaton3}), we need only apply Theorem
\ref{thrm-mainnegativegeneral} and take the product of each
$2$-orbifold with $S^{n-2}$.

\end{proof}


\bibliographystyle{amsplain}

\end{document}